\documentclass[reqno,12pt]{amsart}
\usepackage{amsmath, latexsym, amsfonts, amssymb, amsthm, amscd}
\usepackage{mathrsfs,enumerate}

\numberwithin{equation}{section}

\newtheorem{theorem}{Theorem}[section]
\newtheorem{lemma}[theorem]{Lemma}
\newtheorem{prop}[theorem]{Proposition}

\newtheorem{rem}[theorem]{Remark}
\newtheorem{definition}[theorem]{Definition}

\newcommand{\ga}{\alpha}
\newcommand{\gb}{\beta}

\newcommand{\gep}{\varepsilon}       

\newcommand{\cB}{{\ensuremath{\mathcal B}} }
\newcommand{\cF}{{\ensuremath{\mathcal F}} }
\newcommand{\cP}{{\ensuremath{\mathcal P}} }
\newcommand{\cE}{{\ensuremath{\mathcal E}} }

\newcommand{\cL}{{\ensuremath{\mathcal L}} }

\newcommand{\cS}{{\ensuremath{\mathcal S}} }

\newcommand{\cX}{{\ensuremath{\mathcal X}} }
\newcommand{\cY}{{\ensuremath{\mathcal Y}} }

\newcommand{\bbE}{{\ensuremath{\mathbb E}} }
\newcommand{\E}{{\ensuremath{\mathbb E}} }

\newcommand{\bbN}{{\ensuremath{\mathbb N}} }
\newcommand{\N}{{\ensuremath{\mathbb N}} }

\newcommand{\bbP}{{\ensuremath{\mathbb P}} }

\newcommand{\bbR}{{\ensuremath{\mathbb R}} }
\newcommand{\R}{{\ensuremath{\mathbb R}} }
\newcommand{\bbS}{{\ensuremath{\mathbb S}} }

\newfont{\indic}{bbmss12}

\def\un#1{\hbox{{\indic 1}$_{#1}$}}

\begin{document}

\title{A random string with reflection in a convex domain}

\author{Said Karim Bounebache}
\address{Laboratoire de Probabilit{\'e}s et Mod\`eles Al\'eatoires (CNRS U.M.R. 7599) \\ Universit{\'e} Paris 6
-- Pierre et Marie Curie, U.F.R. Math\'ematiques, Case 188, 4 place
Jussieu, 75252 Paris cedex 05, France }


\maketitle

\begin{abstract}
We study the motion of a random string in a convex domain $O$ in $\R^d$,
namely the solution of a vector-valued stochastic heat equation, confined
in the closure of $O$ and reflected at the boundary of $O$. We study the structure of
the reflection measure by computing its Revuz measure in terms of
an infinite-dimensional integration by parts formula. Our method exploits
recent results on weak convergence of Markov processes with log-concave
invariant measures.
\end{abstract}

\section{Introduction}
In this paper we want to prove well-posedness of stochastic partial differential equations
driven by space-white noise and reflected on the boundary of a convex region of $\R^d$.
More precisely, we consider a convex open domain $O$ in $\R^d$ with a smooth boundary $\partial O$ and a proper l.s.c. convex function $\varphi:\overline O\mapsto\R$,
and we study solutions $(u,\eta)$ of the equation
\begin{equation}\label{1}
\left\{ \begin{array}{ll}
{\displaystyle
\frac{\partial u}{\partial t}=\frac 12
\frac{\partial^2 u}{\partial \theta^2} + n(u(t,\theta)) \cdot \eta(t,\theta) - \frac12\partial\varphi_0(u(t,\theta))
 + \dot W(t,\theta) }
\\ \\
u(0,\theta)=x(\theta), \quad u(t,0)=a, \ u(t,1)=b
\\ \\ {\displaystyle
u(t,\theta)\in\overline O, \ \eta\geq 0, \
\eta(\{(t,\theta) \,|\, u(t,\theta)\notin \partial{O}\})=0 }
\end{array} \right.
\end{equation}
where $u\in C\left([0,T]\times[0,1];\overline O\right)$ and $\eta$ is a
locally finite positive measure on $]0,T]\times[0,1]$; moreover
$a,b\in O$ are some fixed points, $\dot W$ is a vector of $d$ independent
copies of a space-time white noise and for all $y\in \partial O$ we denote by $n(y)$ the inner normal vector at $y$ to the boundary
$\partial O$; finally, $\partial\varphi_0:O\mapsto\R^d$ is the element
of minimal norm in the subdifferential of $\phi$ and
the initial condition $x:[0,1]\mapsto\overline O$ is continuous.

Solutions $u(t,\theta)$ of
equation \eqref{1} take values in the convex closed set $\overline O$ and
evolve as solutions of a standard SPDE in the interior $O$, while the
reflection measure $\eta$ pushes $u(t,\theta)$ along the inner normal vector
$n(u(t,\theta))$, whenever $u(t,\theta)$ hits the boundary.
The condition $\eta(\{(t,\theta) \,|\, u(t,\theta)\notin \partial{O}\})=0$
means that the reflection term acts only when it is necessary, i.e.
only when $u(t,\theta)\in \partial{O}$.

This kind of equations has been considered, in the case of $O$ being
an interval in $\R$, in a number of papers, like \cite{nupa, dopa1,
fuol,za,dmz,deza,debgoud,goud}, as a natural extension
of the classical theory of stochastic differential inclusions in finite dimension
to an infinite-dimensional setting. Moreover, such equations arise
naturally as scaling limit of discrete interface models, see e.g. \cite{fuol}.
However, the finite dimensional situation is very well understood, see
\cite{cepa}, while in infinite
dimension only particular cases can be treated, often with {\it ad hoc} 
arguments.

All previous papers on SPDEs with reflection deal with versions of \eqref{1} where $u$ takes real values, with one or two barriers (one above, one below the solution). This article seems to be the first to tackle the problem of a random string $u$ confined in a convex region in $\R^d$.
This case is not a trivial generalization of the one-dimensional one.
Indeed, in one dimension the reflection term in \eqref{1} has a definite
sign if there is only one barrier, and is the difference of two positive
terms acting on disjoint supports, if there are two barriers. This makes
it easy to obtain estimates on the total variation of the reflection term.
This structure is lost in the
case of a convex region in $\R^d$, since the positive measure $\eta$ is multiplied
by the normal vector $n$ at the boundary, which moves in the $(d-1)$-dimensional
sphere $\bbS^{d-1}$. See the beginning of section \ref{exweso} below for a 
more precise discussion.

In the same spirit, we recall that most of the first papers on this
topic make essential use of monotonicity properties of equation 
\eqref{1}, related with {\it the maximum principle} satisfied by
the second derivative and with the existence of a unique barrier.
However more recent works have shown that monotonicity properties
are not so essential: for instance a fourth-order operator, without
maximum principle, replaces the second derivative in \cite{deza,debgoud,goud},
and two barriers in $\R$ are considered in \cite{eddouk,otobe,debgoud}.

This paper makes use of an approach based on Dirichlet forms, infinite
dimensional integration
by parts formulae, and, crucially, a recent result on stability of
Fokker-Planck equations associated with log-concave reference measures, see 
Theorem \ref{stability} below. This stability result, developed in
\cite{asz} using recent advances in the theory of optimal transport, 
yields convergence of approximating equations to the solution of \eqref{1},
replacing the monotonicity properties used e.g. in \cite{nupa}.
The infinite dimensional integration by parts formula is
with respect to the law of a Brownian bridge conditioned to stay in the domain $O$, proved in \cite{hari}, extending the first formula of this kind, which appeared
in \cite{za}.

We also want to mention that a similar equation, written in the
abstract form of a {\it stochastic differential inclusion}
\begin{equation}\label{spdeH}
dX_t + (AX_t+N_K(X_t))dt \ni dW_t, \qquad X_0=x
\end{equation}
has been considered in \cite{bdt}, where 
$A:D(A)\subset H\mapsto H$ is a self-adjoint positive definite operator
in a Hilbert space $H$,
$K\subset H$ is a closed convex subset with regular boundary,
$N_K(y)$ is the normal cone to $K$ at $y$ and $W$
is a cylindrical Wiener process in $H$.
The authors of \cite{bdt} assume crucially that
$K$ has {\it non-empty interior} in $H$. Our equation \eqref{1} could
be interpreted as an example of \eqref{spdeH} in the framework of \cite{bdt}, where in our case $H=L^2([0,1];\R^d)$ and
\[
K:=\left\{x\in L^2([0,1];\R^d): x_\theta\in \overline O \quad {\rm for \ all} \ \theta\in[0,1]\right\}.
\]
However, in the topology of $L^2([0,1];\R^d)$, $K$ has empty interior
and therefore the approach of \cite{bdt} does not work in our case.
Moreover, our results are somewhat stronger than those of \cite{bdt},
which only deal with the generator and the Dirichlet form rather
than with existence and uniqueness of solutions of the SPDE, as 
we do.

\medskip
The paper is organized as follows. In section 2 we give a precise
definition of solutions to equation \eqref{1}, together with some
notation. In section 4 we introduce the approximating equation and
recall the stability results already mentioned above. In section 5
we prove path continuity of the candidate solution. In section 6 we
state the integration by parts formula we need. In section 7 we prove
existence of weak solutions of equation \eqref{1}, and in section 8
pathwise uniqueness and existence of strong solutions. Finally, in
section 9 we prove some properties of the reflection measure $\eta$.

\section{Notations and setting}

 We first discuss the notion of solution of \eqref{1}. We consider a convex l.s.c.
 $\varphi:\overline O\mapsto[0,+\infty]$ such that $\varphi<+\infty$ on $O$.
 We denote 
by $D(\varphi):=\{\varphi<+\infty\}$ the {\it domain} of $\varphi$ and by 
$\partial \varphi$
 the {\it subdifferential} of $\varphi$:
\[
\partial \varphi(y):=\left\{ z\in \R^d: \, \varphi(w)\geq \varphi(y)+\langle z,w-y\rangle,
\ \forall \, w\in \overline O\right\}, \qquad y\in D(\varphi).
\]
The set $\partial \varphi(y)$ is non-empty, closed and convex in $\R^d$, and therefore it has
a unique element of minimal norm, that we call $\partial_0 \varphi(y)$. 
Notice that we do not assume smoothness of $y\mapsto \partial_0 \varphi(y)$.
We can also allow $\partial_0 \varphi(y)$ to blow up as $y\to\partial O$, but
not too fast. Indeed, throughout
the paper we assume
that $\partial_0 \varphi:D(\varphi)\mapsto\R^d$ satisfies
\begin{equation}\label{assumphi}
\int_{O} |\partial_0 \varphi(y)|^2 \, dy<+\infty
\end{equation}
where $dy$ denotes the Lebesgue measure on $O$. This assumption
is not optimal, see Remark \ref{nonoptimal} below, but already covers interesting cases, like logarithmic divergences or polynomial divergences with small exponent,
see \cite{debgoud} or \cite{za2} for related studies in convex subsets of $\R$.

For two vectors $a,b\in\R^d$, we denote by $a\cdot b$ their canonical
scalar product.
We consider the Hilbert space $H:=L^2([0,1];\R^d)$, 
endowed with the canonical scalar product $\langle\cdot,\cdot\rangle$ and
norm $\|\cdot\|$,
\[
\langle h,k\rangle:=\int_0^1 h(\theta)\cdot k(\theta)\, d\theta, \qquad
\|h\|^2:=\langle h,h\rangle, \qquad h,k\in H.
\]
\begin{definition}
Let $x \in C\left([0,1];\overline O\right)$. An adapted triple $(u,\eta, W)$, defined on a complete filtered probability space $(\Omega, \cF, (\cF_t)_t,\bbP)$, is a weak solution of \eqref{1} if 
\begin{itemize}
\item a.s. $u\in C(]0,T]\times[0,1];\overline O)$ and $\E[\|u_t-x\|^2]\to 0$
as $t\downarrow 0$
%
\item a.s. $\eta$ is a positive measure on $]0,T]\times[0,1]$ such that
$\eta([\gep,T]\times[0,1])<+\infty$ for all $0<\gep\leq T$
\item a.s. the function $(t,\theta)\mapsto |\partial_0 \varphi(u(t,\theta))|$
is in $L^1_{\rm loc}([\gep,T]\times\,]0,1[)$ for all $0<\gep\leq T$
\item $W=(W^1,\ldots,W^d)$ is a vector of d independent copies of a Brownian sheet
\item for all $h \in C^2_c((0,1);\R^d)$ and $0<\gep\leq t$
\begin{equation}\label{wweak}
\begin{split}
\langle u_t-u_\gep, h\rangle = & \ 
\frac12\int_\gep^t \langle h'', u_s \rangle \,ds +  \int_\gep^t \int_0^1 h(\theta)
\cdot n(u(s,\theta)) \, \eta(ds,d\theta) \\  
  &  - \frac 12\int_\gep^t \langle  h, \partial_0\phi(u_s) \rangle \, ds
+ \int_\gep^t \int_0^1 h(\theta)\,W(ds,d\theta)
\end{split}
\end{equation}
\item a.s. the support of $\eta$ is contained in 
$\{(t,\theta): u(t,\theta)\in\partial O\}$, i.e.
\begin{equation}\label{contact}
\eta(\{(t,\theta) \,|\, u(t,\theta)\notin \partial{O}\}) = 0.
\end{equation}
\end{itemize}
A weak solution $(u,\eta, W)$ is said to be a strong solution if $(u,\eta)$ is
adapted to the natural filtration of $W$.
\end{definition}
We say that pathwise uniqueness holds for equation \eqref{1}
if any two weak solutions $(u^1,{\eta}^1,W)$ and $(u^2,{\eta}^2,Z)$ 
coincide. 
In this article we want to prove the following result:
\begin{theorem}\label{main}
For all  $x \in C\left([0,1];\overline O\right)$, the problem \eqref{1} enjoys pathwise uniqueness of weak solutions and
existence of a strong solution. 
\end{theorem}
Next, we want to study some properties of the reflection measure $\eta$.
We recall that its support is contained in the contact set, i.e.
in the set $\{(t,\theta): u(t,\theta)\in\partial O\}$. 
The next result shows that $\eta$ is concentrated on a subset $\cS$
of the contact set, such that each section $\cS\cap(\{s\}\times[0,1])$,
$s\geq 0$, contains at most one point. Moreover, $u(s,\cdot)$ hits the
boundary $\partial O$ at this point and not elsewhere.
\begin{theorem}\label{main2}
A.s. the reflection measure
$\eta$ is supported by a Borel set $\cS\subset\,]0,+\infty[\, \times[0,1]$, i.e.
$\eta(\cS^c)=0$, such that
for all $s\geq 0$, the section $\{\theta\in[0,1]: (s,\theta)\in \cS\}$
has cardinality $0$ or $1$. Moreover, if $r(s)\in \cS\cap(\{s\}\times[0,1])$
then
\[
u(s,r(s))\in\partial O, \qquad u(s,\theta)\notin\partial O, \ 
\forall \theta\in[0,1]\setminus\{r(s)\}.
\]
\end{theorem}
This property is analogous to that discovered in \cite{za} for
reflected SPDEs in $[0,+\infty)$. We recall that in this one-dimensional
setting, sections of the contact set have been studied in detail in 
\cite{dmz}. It would be very interesting to prove the same kind of results
in our multi-dimensional setting.

\subsection{Notations} We fix now some notations which will be used
throughout the paper. 
Let $E:=H^{[0,\infty)}$ and define the canonical process
$X_t:E\mapsto H$, $t\geq 0$, $X_t(e):=e(t)$, and the associated
natural filtration
\[
\cF_\infty^0 \ := \ \sigma\{X_s,\, s\in[0,\infty)\}, \quad \cF_t^0 \ := \ \sigma\{X_s,\, s\in[0,t]\}, \quad t\in[0,+\infty].
\]
We 
denote by $\mu$ the law of the Brownian bridge from $a$ to $b$ in $\R^d$. 
Let us define 
\[
K:=\{x\in L^2([0,1];\R^d): x_\theta\in \overline O, \ {\rm for \ all} \ \theta\in[0,1]\}
\]
and for all $x\in H=L^2([0,1];\R^d)$ we define $U:H\mapsto [0,+\infty]$ as follows
\[
U(x):= \left\{ \begin{array}{ll}
{\displaystyle \int_0^1 \varphi(x_\theta)\, d\theta, \quad {\rm if} \ 
x\in K
}
\\ \\
+\infty, \qquad {\rm otherwise}.
\end{array} \right.
\]
\begin{lemma}
The probability measure $\nu$ on $K$
\begin{equation}\label{nu}
\nu(dx) := \frac1Z \, \exp(-U(x)) \, \mu(dx).
\end{equation}
is well defined, i.e. $\mu(K)>0$ and $Z:=\mu(e^{-U})\in\,]0,1]$.
\end{lemma}
\begin{proof}
Since $a,b\in O$ and the Brownian bridge has continuous paths, 
$\mu(K)$ is clearly positive. Analogously, if $O_\delta$ is the subset of $O$ of all elements with distance 
greater than $\delta>0$ from $\partial O$, then the convex function $\phi$ is bounded on 
$O_\delta$. Therefore $U$ is bounded on $K_\delta:=\{x\in L^2([0,1];\R^d): x_\theta\in O_\delta, \ {\rm for \ all} \ \theta\in[0,1]\}$
and $Z\geq \mu(e^{-U}\un{O_\delta})>0$.
\end{proof}
We note that $U$ is l.s.c. and convex.
For the next Lemma,  see \cite[Chapter 2]{bre}.
\begin{lemma}[Yosida approximation]\label{yosid}
Let $\Phi:\R^d\mapsto\R\cup\{+\infty\}$ be convex lower semi-continuous, and $\partial \Phi$ be the subdifferential of $\Phi$. Set for $n\in\N$
\[
\Phi_n(x):=\inf_{y\in \R^d}\left\{\Phi(y)+n\,\|x-y\|^2\right\}, \qquad
x\in \R^d.
\]
Then
\begin{enumerate}
\item $\partial \Phi_{n}$ is $n$-Lipschitz continuous
\item $\forall y \in D(\Phi)$,
$\Phi_{n}(y) \uparrow \Phi(y)$, as  $n \uparrow +\infty$
\item $\forall y \in D(\partial \Phi_{n})$,
\begin{center}
               $\lim\limits_{n \to +\infty} \partial \Phi_{n}(y)$ $=$ $\partial_0 \Phi(y)$, and $\vert \partial \Phi_{n}(y) \vert \uparrow \vert \partial_{0}\Phi(y)\vert$ as $n \rightarrow +\infty$
\end{center}
\end{enumerate}
\end{lemma}
Then we define $U_n:H\mapsto [0,+\infty)$ as follows
\begin{equation}\label{U_n}
U_n(x):= \int_0^1 \Phi_n(x_\theta)\, d\theta, \quad {\rm if} \ 
x\in H.
\end{equation}
\begin{rem}\label{nonoptimal}{\rm
The assumption \eqref{assumphi} on $\phi$ is far from optimal.
In fact, our approach covers a more general class of non-linearity;
indeed, the proof we give below yields Theorems \ref{main} and \ref{main2}
under the assumption
\begin{equation}\label{assumphi2}
\int \int_\delta^{1-\delta} |\partial_0 \varphi(x_\theta)|^2 \, d\theta
\, \nu(dx)<+\infty, \qquad \forall \delta\in(0,1/2),
\end{equation}
see Lemma \ref{nuphi} below.
}
\end{rem}

Finally, we need to introduce some function spaces. We denote by $C_b(H)$ the
Banach space of all $\varphi:H\mapsto{\mathbb R}$ being bounded
and continuous in the norm of $H$, endowed with the norm
$\|\varphi\|_\infty:=\sup |\varphi|$. Moreover we 
denote by $\mathcal{FC}^1$ the set of all functions
 $F$ of the form
\begin{equation}
F(w) = f(\langle l_1, w\rangle, ..., \langle l_n, w\rangle), \qquad w\in H,
\end{equation}
with $n\in \bbN$, $l_i \in L^2(0,1)$ and $f\in C_b^1(\bbR^n)$.

\section{The approximating equation}
Let us introduce the convex function
$\Phi:\R^d\mapsto\R\cup\{+\infty\}$
\begin{equation}\label{Phi}
\Phi(y):=\left\{ \begin{array}{ll}
{\displaystyle \varphi(y), \quad {\rm if} \
y\in \overline O
}
\\ \\
+\infty, \qquad {\rm otherwise}
\end{array} \right.
\end{equation}
and its Yosida approximation $\Phi_n$, defined as in Lemma \ref{yosid}. 
We introduce the SPDE
\begin{equation}\label{appro}
\left\{ \begin{array}{ll}
{\displaystyle
\frac{\partial u_n}{\partial t}=\frac 12
\frac{\partial^2 u_n}{\partial \theta^2} - \frac 12\partial\Phi_n(u_n(t,\theta))
 + \dot W(t,\theta) }
\\ \\
u_n(0,\theta)=x(\theta), \ u_n(t,0)=a, \ u_n(t,1)=b
\end{array} \right.
\end{equation}
By Lemma \ref{yosid}, $\partial \Phi_n$ is Lipschitz continuous,
and therefore it is a classical result 
that for any $x\in H$ equation \eqref{appro} has a unique solution $u_n$,
which is moreover a.s. continuous on $]0,\infty[\, \times[0,1]$. 

Equation
\eqref{appro} is a natural approximation of equation \eqref{1} and one
expects $u_n$ to converge to $u$ in some sense as $n\to\infty$.
A convergence in law indeed holds and follows from a general result
proven in \cite{asz}, see the discussion in Theorem \ref{stability}
below. 

We denote by $\bbP^n_x$ the law on $E=H^{[0,\infty)}$ of $(u^n(t,\cdot))_{t\geq 0}$,
solution of \eqref{appro}. We also define the probability measure
\begin{equation}\label{nu^n}
\nu_n(dx) := \frac1{Z_n} \, \exp(-U_n(x)) \, \mu(dx), \qquad
U_n(x):=\int_0^1 \Phi_n(x_\theta)\, d\theta,
\end{equation}
and the symmetric bilinear form $(\cE^n, \mathcal{FC}^1)$
\begin{equation}\label{e^n}
\cE^n(F,G):=\frac12\int \langle\nabla F,\nabla G\rangle \, d\nu_n, \qquad F,G\in \mathcal{FC}^1.
\end{equation}
We denote by $(P^n_t)_{t\geq 0}$ the transition semigroup associated to
equation \eqref{appro}:
\[
P^n_t\varphi(x):=\E^n_x(\varphi(X_t)), \qquad \forall \, \varphi\in C_b(H),
\ x\in H, \ t\geq 0,
\]
and the associated resolvent
\[
R^n_\lambda\varphi(x):=\int_0^\infty e^{-\lambda t} \,
{\mathbb E}_x^n \, [\varphi(X_t)]\, dt, \qquad
x\in H, \ \lambda>0.
\]
The following result is well known, see \cite{maro} and \cite{dpz2}.
\begin{theorem}\label{apprteo} \ 
\begin{enumerate}
\item $({\cE}^n,\mathcal{FC}^1)$ is
closable in $L^2(\nu_n)$: we denote
by $({\cE}^n,D({\cE}^n))$
the closure. 
\item $(\bbP^n_x)_{x\in H}$ is a Markov process,
associated with the Dirichlet form $(\cE^n,D(\cE^n))$ in $L^2(\nu_n)$, i.e.
for all $\lambda>0$ and $\varphi\in L^2(\nu_n)$, 
$R_\lambda^n\varphi\in D({\cE}^n)$ and:
\[
\lambda \int_H R_\lambda^n\varphi \, \psi \,
d\nu_n + 
{\cE}^n(R^n_\lambda
\varphi,\psi) \, = \, \int_H \varphi \, \psi \, d\nu_n,
\quad \forall\psi\in D({\cE}^n).
\]
\item $\nu_n$ is the unique
invariant probability measure of $(P^n_t)_{t\geq 0}$. Moreover, 
$(P^n_t)_{t\geq 0}$ is symmetric with respect to
$\nu_n$.
\end{enumerate}
\end{theorem}
We recall an important property of equation \eqref{appro}: the associated
transition semigroup $(P^n_t)_{t\geq 0}$ is {\it Strong Feller}, i.e. $P_t^n$ maps bounded Borel functions into bounded continuous
functions
for all $t>0$. Indeed, $P_t^n$ satisfies for any bounded Borel $\varphi:H\mapsto\R$
\begin{equation}\label{sfellen}
|P^n_t\varphi(x)-P^n_t\varphi(y)| \, \leq \, 
\frac{\|\varphi\|_\infty}{\sqrt t} \, \|x-y\|_H, \qquad x,y\in
H, \ t>0,
\end{equation}
see \cite[Proposition 4.4.4]{sc}.

\section{The Dirichlet Form}
One of the main tools of this paper is the Dirichlet form associated with
equation \eqref{1}. Recall the definition \eqref{nu}
of the probability measure $\nu$.
Notice that $\mu$ is Gaussian and $U$ is convex. It follows that
$\nu$ is {\it log-concave}, i.e. for all pairs of open sets $A,B \subset H$  we have:
\[
\log\nu((1-t)A+tB) \geq (1-t)\log\nu(A)+t\log\nu(B) 
\] 
where $(1-t)A+tB := \{(1-t)a+tb \,|\, a\in A,b\in B\}$ for $t \in [0,1]$,
see for instance \cite[Theorem 9.4.11]{ags}. Notice that $\nu_n$ defined in 
\eqref{nu^n} above is also log-concave for the same reason.

Let us consider now the bilinear form
\begin{equation}\label{e^}
\cE(F,G):=\frac12\int \langle\nabla F,\nabla G\rangle \, d\nu, \qquad F,G\in \mathcal{FC}^1.
\end{equation}
Then by \cite[Theorem 1.2]{asz}
\begin{theorem}\label{exmarkov} In the previous setting we have:
\begin{enumerate}
\item The bilinear form $(\cE, \mathcal{FC}^1)$ is closable in $L^2(\mu)$ and its closure 
$(\cE,D(\cE))$ is a Dirichlet form.
\item 
There is a Markov family $(\bbP_x)_{x\in K}$ of probability measures on the 
canonical path space $(K^{[0,+\infty[},\cF,(\cF_t), (X_t)_{t\geq 0})$ associated with $\cE$.
\item for all $x\in K$, $\bbP_x$-a.s. $(X_t)_{t> 0}$ is continuous in $H$ and
$\E_x[\|X_t-x\|^2]\to 0$ as $t\to 0$.
\end{enumerate} 
\end{theorem} 
Let us remark that clearly, since $U_n\uparrow U$,
we have
\begin{equation}\label{nu^ntonu}
\nu_n\rightharpoonup \nu.
\end{equation}
A look at the Dirichlet forms \eqref{e^n} and \eqref{e^} suggests that
the laws of the associated processes could also
converge. In general this is false, and a number of papers have been
devoted to this problem, see for instance \cite{kushi} and \cite{koles}.
However, it turns out that, in the setting of Dirichlet forms of the form \eqref{e^n} with log-concave reference measures, \eqref{nu^ntonu}
does imply convergence in law of the associated Markov processes.
This general {\it stability property} is one of the main results of \cite{asz}. By \eqref{nu^ntonu} and \cite[Theorem 1.5]{asz} we have that
\begin{theorem}[Stability]\label{stability}
For any sequence $x_n\in H$ converging 
to $x\in K$, we have that 
\begin{itemize}
\item[(a)] $\bbP_{x_n}^n\to\bbP_x$ weakly in $H^{[0,+\infty[}$ as $n\to\infty$,
\item[(b)] for all $0<\varepsilon\leq
T<+\infty$, $\bbP_{x_n}^n\to\bbP_x$ weakly in $C([\varepsilon,T];H_w)$,
\item[(c)] for all $0\leq T<+\infty$, $\bbP_{\nu_n}^n\to\bbP_{\nu}$ weakly in
$C([0,T];H_w)$,
\end{itemize}
where $H_w$ is $H$ endowed with the weak topology and 
\[
\bbP_{\nu_n}^n = \int \bbP_{y}^n\, \nu_n(dy), \qquad
\bbP_{\nu} = \int \bbP_{y}\, \nu(dy).
\]
\end{theorem}
This stability result will be very useful to prove several properties of the solution to \eqref{1}. We notice
that, by point (a) of Theorem \ref{stability},
\begin{equation}\label{convsemi}
\lim_{n\to\infty} P^n_t\varphi(x)=P_t\varphi(x):=
\E_x(\varphi(X_t)), \qquad \forall \, t>0, \ x\in K, \ \varphi\in C_b(H).
\end{equation}
This already allows to draw an important consequence of Thorem \ref{stability}.

\begin{prop}\label{absocon}
$ $
\begin{itemize}
\item
The Markov semigroup $(P_t)_{t\geq 0}$
associated with the Dirichlet form $(\cE,D(\cE))$ is Strong Feller, i.e.
for any bounded Borel $\varphi:H\mapsto\R$
\begin{equation}\label{sfelle}
|P_t\varphi(x)-P_t\varphi(y)| \, \leq \, 
\frac{\|\varphi\|_\infty}{\sqrt t} \, \|x-y\|_H, \qquad x,y\in
H, \ t>0.
\end{equation}
\item The Markov process  $(\bbP_x)_{x\in K}$ associated with 
$(\cE,D(\cE))$ satisfies the absolute continuity condition: the
transition probability $p_t(x,\cdot)=\bbP_x(X_t\in\cdot)$ is
absolutely continuous w.r.t. the invariant measure $\nu$
\begin{equation}\label{eq:absocon}
p_t(x,\cdot)\ll \nu(\cdot), \qquad x\in H, \ t>0.
\end{equation}
\end{itemize}
\end{prop}
\begin{proof}
The first point follows from \eqref{sfellen} and the weak convergence
result of Theorem \ref{stability}-(a). The second claim follows from
the first. Indeed, let $A\subset H$ be a Borel set with $\nu(A)=0$.
Then for all $t>0$, by invariance $\nu(P_t\un{A}) = \nu(A)=0$, i.e.
$P_t\un{A}(x)=0$ for $\nu$-a.e. $x\in H$. But by the Strong Feller
property $P_t\un{A}$ is continuous on $K$, therefore we obtain that
$P_t\un{A}(x)=0$ for all $x$ in the support of $\nu$, which coincides with $K$.
\end{proof}

Finally, we give a result on existence and uniqueness of
invariant measures of \eqref{1}.
\begin{prop}\label{measinva}
There exists a unique invariant probability measure of
the Markov semigroup $(P_t)_{t\geq 0}$, and it is equal to $\nu$.
\end{prop}
\begin{proof}
It is well known that $\nu_n$ is an invariant probability measure of
the Markov semigroup $(P_t^n)_{t\geq 0}$. The weak convergence 
of $\nu_n$ to $\nu$, the convergence formula \eqref{convsemi} of $P^n_t$
to $P_t$ and the Strong Feller property, uniform in $n$, of $P^n_t$
allow to show that $\nu$ is invariant for $(P_t^n)_{t\geq 0}$.

To prove uniqueness, we use a coupling argument. Let 
$m^1$ and $m^2$ be two invariant probability measures for $(P_t)_{t\geq 0}$ 
and let $q_1$ and $q_2$ be $K$-valued random
variables, such that the law of $q^i$ is $m^i$ and 
$\{q^1,q^2,W\}$ is an independent family. Let $u^n_i$ the solution of 
equation \eqref{appro} with $u^n_i(0,\cdot)=q_i$, $i=1,2$.
Setting $v:=u^n_1(t,\cdot)-u^n_2(t,\cdot)$, 
we have:
\[
\frac d{dt} \|v\|^2 \, = -\|v'\|^2-\langle u^n_1(t,\cdot)-
u^n_2(t,\cdot), \partial \Phi_n(u^n_1(t,\cdot))-
\partial \Phi_n(u^n_2(t,\cdot)) \rangle \leq  -\pi^2\|v\|^2 
\]
since $\langle p-q, \partial \Phi_n(p)-
\partial \Phi_n(q) \rangle \geq 0$ by convexity of $\Phi_n$.
Therefore for all $n$
\[
\|u^n_1(t,\cdot)-u^n_2(t,\cdot)\| \, \leq \,
e^{-\pi^2 t/2}\|q_1-q_2\|, \quad \forall t\geq 0.
\]
Passing to the limit as $n\to\infty$, we obtain by Theorem \ref{stability} 
that $(u^n_1,u^n_2,W)$ converges in law as $n\to\infty$
to $(u_1,u_2,W)$, where $(u_i,W)$ is a weak solution of \eqref{1}  
with $u_i(0,\cdot)=q_i$, $i=1,2$. Then we obtain that
\[
\|u_1(t,\cdot)-u_2(t,\cdot)\| \, \leq \,
e^{-\pi^2 t/2}\|q_1-q_2\|, \quad \forall t\geq 0.
\]
Since the law of $u_i(t,\cdot)$ is equal to $m^i$ for all $t\geq 0$,
this implies $m^1=m^2$. 
\end{proof}

\section{Continuity properties of $X$}

From the general theory of \cite{asz} and Theorem \ref{exmarkov} above,
one obtains only relatively mild continuity path properties
of $X$, namely continuity in $t$ with values in $L^2(0,1)$. 
However, for the contact condition \eqref{contact} to make sense,
we need $u(t,\cdot)=X_t$ to be jointly continuous, since we need to
evaluate $u$ at points $(t,\theta)\in[0,T]\times[0,1]$. This is the content
of the main result of this section
\begin{prop}\label{continuity}
For any $x\in K$, there exists a modification of $u(t,\cdot)=X_t$ which is
$\bbP_x$-a.s. continuous on $]0,+\infty[\, \times[0,1]$ and such that
$\E_x[\|u_t(\cdot)-x\|^2]\to 0$ as $t\to 0$.
\end{prop}
We start by proving continuity of stationary solutions of \eqref{1}.
To this aim,
we are going to use the approximating equations \eqref{appro} and the
convergence result of Theorem \ref{stability}-(c) for the stationary solutions.
In particular, we are going to prove tightness of $(\bbP_{\nu_n}^n)_n$
in $C([0,1]\times [0,T])$.
\begin{lemma}\label{tight}
The sequence $(\bbP_{\nu_n}^n)_n$ is tight in $C([0,1]\times [0,T])$.
\end{lemma}
\begin{proof} We follow the proof of Lemma 5.2 in \cite{deza}.
We first recall a result of \cite[Th. 7.2 ch 3]{ek}.
Let $(P,d)$ be a Polish space, and let $(X_\alpha)_\alpha$ be a family of processes with sample paths in $C([0,T];P)$. Then the laws of $(X_\alpha)_\alpha$ are relatively compact if and only if the following two conditions hold:
\begin{enumerate}
\item For every $\eta >0$ and rational $t\in[0,T]$, there is a compact set $\Gamma_\eta^t \subset P$ such that:
\begin{equation}\label{tension0}
\inf\limits_\alpha \bbP\left( X_\alpha \in \Gamma_\eta^t\right) \geq 1-\eta
\end{equation} 
\item For every $\eta,\epsilon >0$ and $T>0$, there is $\delta>0$ such that 
\begin{equation}\label{tension}
\sup\limits_\alpha \bbP\left( w(X_\alpha, \delta, T)\geq \epsilon \right) \leq \eta 
\end{equation} 
 \end{enumerate}
where $w(\omega, \delta, T):=\sup\{d(\omega(r),\omega(s)): r,s\in[0,T],
\, |r-s|\leq\delta\}$ is the modulus of continuity in $C([0,T];P)$.

\medskip\noindent We introduce the space $H^{-1}(0,1)$, 
completion of $L^2(0,1)$ w.r.t. the norm:
\[
\|f\|_{-1}^2 \, := \, \sum_{k=1}^\infty k^{-2} \,
|\langle f, e_k \rangle_{L^{2}(0,1)}|^2
\]
where $e_k(r):={\sqrt 2}\sin(\pi kr)$, $r\in[0,1]$, $k\geq 1$, are the eigenvectors of the second derivative with
homogeneous Dirichlet boundary conditions at $\{0,1\}$. Recall that $L^{2}(0,1)=H$, in
our notation.
We denote by $\kappa$ the Hilbert-Schmidt norm of
the inclusion $H \to H^{-1}(0,1)$, which by definition is equal in our case to
\[
\kappa = \sum_{k\geq 1} k^{-2}<+\infty.
\]
We claim that for all $p>1$ there exists $C_p\in(0,\infty)$, independent of
$n$, such that:
\begin{equation}\label{estim}
\left( {\mathbb E}\left[
\left\|X^{n}_t- X^{n}_s\right\|^p_{H^{-1}(0,1)} \right]
\right)^{\frac1p} \, \leq \, C_p \, |t-s|^{\frac12}, \qquad
t,s\geq 0.
\end{equation}
To prove (\ref{estim}),
we fix $n>0$ and $T>0$ and use the Lyons-Zheng
decomposition, see e.g. \cite[Th. 5.7.1]{fot}, to write for
$t\in[0,T]$ and $h\in H$:
\[
\langle h,X^{n}_t - X^{n}_0\rangle_H
\, = \, \frac 12 \,
M_t \, - \, \frac 12 \, (N_T  - N_{T-t}), 
\]
where $M$, respectively $N$, is a martingale w.r.t. the natural
filtration of $X^{n}$, respectively of $(X^{n}_{T-t},
\ t\in[0,T])$.
Moreover, the quadratic variations are both equal to:
$\langle M\rangle_t =\langle N\rangle_t =t \cdot \|h\|^2_H$. By the
Burkholder-Davis-Gundy inequality we can find $c_p\in(0,\infty)$
for all $p>1$ such that: $\left({\mathbb E}\left[|\langle X^{n}_t-
X^{n}_s, e_k \rangle|^p \right]\right)^{\frac1p} \leq c_p \,|t-s|^{\frac12}$,
$t,s\in[0,T]$, and therefore
\[
\begin{split}
& \left({\mathbb E}\left[\left\|X^{n}_t-
X^{n}_s\right\|^p_{H^{-1}(0,1)} \right]
\right)^{\frac1p}  \leq  \sum_{k\geq 1} k^{-2}\left({\mathbb E}\left[|\langle X^{n}_t-
X^{n}_s, e_k \rangle|^p \right]
\right)^{\frac1p} \\
  & \leq  c_p \sum_{k\geq 1} k^{-2} |t-s|^{\frac12} \|e_k \|^{2}_{L^{2}(0,1)}
  \leq  c_p \, \kappa \,|t-s|^{\frac12}, \quad t,s\in[0,T],
\end{split}
\]
and (\ref{estim}) is proved. 
Let us introduce now the norm $\|\cdot\|_{W^{\eta,r}(0,1)}$ for $\eta>0$, $r\geq 1$
\[
\|x\|_{W^{\eta,r}(0,1)}^r = \int_0^1 |x_s|^r ds + \int_0^1\int_0^1 \dfrac{|x_s - x_t|^r }{|s-t|^{r\eta+1}} \, dt \, ds.
\]
By stationarity 
\begin{align}\label{estim2}
\nonumber
& \left( {\mathbb E}\left[\left\|X^{n}_t-
 X^{n}_s\right\|^p_{W^{\eta,r}(0,1)} 
\right]\right)^{\frac1p}
\leq  \left( {\mathbb E}\left[\left\|
 X^{n}_t\right\|^p_{W^{\eta,r}(0,1)} \right] 
\right)^{\frac1p}+\left( {\mathbb E}\left[\left\|
 X^{n}_s\right\|^p_{W^{\eta,r}(0,1)} \right]
\right)^{\frac1p}
\\ \nonumber \\  & \qquad
= 2\left( \int_H \| x\|^p_{W^{\eta,r}(0,1)} \, d\nu_n \right)^{\frac1p}
\, \leq \, c \left( \int_H \| x\|^p_{W^{\eta,r}(0,1)} \,d\mu \right)^{\frac1p}
\end{align}
since $U\geq U_n\geq 0$, where $c=Z^{-1/p}$. 
If $r>p\geq 1$ the Jensen inequality for a concave function gives us,
for  $\eta\in(0,1/2)$,
\[
\begin{split}
\left( \E\left( \| \beta\|^p_{W^{\eta,r}(0,1)} \right) \right)^{\frac{r}{p}} &  \leq  \E  \left(\|\beta\|_r^p + \int_0^1\int_0^1 \dfrac{|\beta_s - \beta_t|^r }{|s-t|^{r\eta+1}} \,dt\,ds\right)
\\ & \leq 1+ \int_0^1\int_0^1 {|s-t|^{r(\frac12-\eta)-1}} \,dt\,ds <  +\infty.
\end{split}
\]
The latter term is finite since $\mu$ is the law of a Brownian bridge. Let us now 
fix any $\eta\in(0,1/2)$ and $\gamma\in(0,1)$ such that 
\[
\gamma>\frac1{1+\frac23\eta}.
\]
From this it follows that $\alpha:=\gamma\eta-(1-\gamma)>0$ and 
therefore, if $r>0$ is such that
\[
r>\max\left\{\frac2{1-\gamma}, \frac1{\eta-\frac32\frac{1-\gamma}\gamma}
\right\},
\]
then we obtain that
\[
\frac{r}2 \, (1-\gamma)>1, \qquad \frac 1d:=\gamma \frac 1r + (1-\gamma)\frac12
<\alpha.
\]
Then by interpolation, see \cite[Chapter 7]{af},
\begin{eqnarray*}
& &
\left( {\mathbb E}\left[\left\| X^{n}_t-
X^{n}_s\right\|^p_{W^{\ga,d}(0,1)} \right]
\right)^{\frac1p}\leq 
\\ \\ & & \leq \,
\left( {\mathbb E}\left[\left\| X^{n}_t-
X^{n}_s\right\|^p_{W^{\eta,r}(0,1)} \right]
\right)^{\frac \gamma p}\left( {\mathbb E}\left[\left\|
X^{n}_t- X^{n}_s\right\|^p_{H^{-1}(0,1)}
\right] \right)^{\frac{1-\gamma}p}.
\end{eqnarray*}
Since $\alpha d>1$, there exists
$\beta>0$ such that $(\ga-\gb)d>1$. By the Sobolev embedding,
$C^{\gb}([0,1])\subset W^{\ga,d}(0,1)$. Since $\frac{r}2 \, (1-\gamma)>1$,
there is $1<p<r$ such that $\frac{p}2 \, (1-\gamma)=1+\zeta>1$,
and by (\ref{estim}) and (\ref{estim2}), we find that
\[
\left( {\mathbb E}\left[\left\| X^{n}_t-
X^{n}_s\right\|^p_{C^{\gb}([0,1])} \right]
\right) \, \leq \, {\tilde c} \,
|t-s|^{\frac{1-\gamma}{2}p} .
\]
We consider now, as Polish space $(P,d)$, the Banach space
$C^{\gb}([0,1])$. By Kolmogorov's criterion, see e.g. \cite[Thm. I.2.1]{reyo},
we obtain that a.s. $w(X^{n},\delta, T)\leq C \, \delta^{\frac{\zeta}{2p}}$, with $C\in L^p$. Therefore by the Markov inequality, if $\epsilon>0$
\[
\bbP \left(w(X^{n},\delta, T)\geq \epsilon \right)  \leq \bbE \left[ {C^{p}} \right]
\, \delta^{\frac\zeta2}\epsilon^{-p},
\]
and \eqref{tension} follows for $\delta$ small enough.

Finally, since for all $t\geq 0$ the law of $X^{n}_t$ is
$\nu_n$, which converges as $n\to \infty$ weakly in $C([0,1])$,
tightness of the laws of $(X^{n})_{n>0}$
in $C([0,T]\times [0,1])$ and therefore \eqref{tension0} follow.
\end{proof}

\begin{proof}[Proof of Proposition \ref{continuity}]
By Theorem \ref{stability}, we have $\bbP^n_{\nu_n}\rightharpoonup \bbP_\nu$
in $C_b([0,T];H_w)$ and by Lemma \ref{tight} the sequence $(\bbP^n_{\nu_n})_n$
is tight in $C([0,T]\times[0,1])$, so that $\bbP_\nu(C([0,T]\times[0,1]))=1$.
Now, we want to prove that $\bbP_x(C(]0,T]\times[0,1]))=1$ for all $x\in K$.
Let $\gep>0$. By \eqref{eq:absocon}, $\bbP_x\ll \bbP_\nu$ over the $\sigma$-algebra
$\sigma\{X_s, s\geq\gep\}$. Therefore $\bbP_x(C([\gep,T]\times[0,1]))=1$ for
all $\gep>0$.
\end{proof}

\section{An integration by parts formula}

An important tool in the construction of a solution to equation \eqref{1}
is the following integration by parts formula on the law
$\mu$ of the Brownian bridge on the set $K$ of trajectories contained in $\overline O$,
proved in \cite[Theorem 1.1]{hari}: 
\begin{equation}\label{hariya}
\int_K \partial_h F \, d\mu = -\int_K \langle h'',x\rangle F\, d\mu -
\int_{\partial O}\sigma(dy) \, \E_{a,y,b}\left[n(y)\cdot h(S_w)\, F(w)\right] \, \lambda(y)
\end{equation}
where $F\in \mathcal{FC}^1$ and
\begin{enumerate}
\item $h$ is in the Cameron-Martin space  of $\mu$ 
\[
H_0^1=\left\{h \in C^0\, | \, h_0=h_1=0, \ h_t=\int_0^t \dot{h}_s \, ds, \ \dot{h}\in L^2(0,1)\right\}
\]
\item 
$\bbP_{a,y,b}$ is the law of two independent Brownian motions put together back to back at their first exit time of $O$, across $y$. More precisely, let $B$ and $\hat{B}$ be two independent Brownian motion such that $B_0 = a$ and $\hat{B}_0 = b$. Let $\tau(B)$ and $\tau(\hat{B})$ be the first exit times from $O$ of $B$ and $\hat{B}$ respectively. Conditionally on  $\tau(B)+\tau(\hat{B})=1$, $B_{\tau(B)}=y$ and $\hat{B}_{\tau(\hat{B})}= y$, define the process $X$ by
\[
X_t = \left\{ \begin{array}{ll}
B_t & 0\leq t \leq \tau(B)\\
\hat{B}_{\tau(B)+\tau(\hat{B}) - t}, & \tau(B)\leq t \leq \tau(B)+\tau(\hat{B})
\end{array}
\right.
\]  
Then $X$ has the law $\bbP_{a,y,b}$. For $w\in C([0,1];\overline O)$ we denote by $S_w$ the first time at which $w_{S_w}\in\partial O$, if there is any:
\[
S_w:=\inf\{s\in[0,1]: \, w_s\in\partial O\}, \qquad \inf\emptyset:=0.
\]
Then $w_{S_w}=y$ for $\bbP_{a,y,b}$-a.e. $w$.
\item $\sigma$ is the surface measure on $\partial O$, $n_y$ is the inward normal vector
\item $(p_t(x,y))_{t>0,x,y\in O}$ is the fundamental solution to the Cauchy problem
\[
\frac\partial{\partial t} - \dfrac{1}{2}\Delta = 0
\]
where $\Delta$ is the Laplace operator on $O$ with homogeneous Dirichlet boundary conditions at $\partial O$, and
\[
 \lambda(y):=  \frac1{2p_{1}(a,b)}
\int_0^1\dfrac{\partial}{\partial n_y}\, p_u(a,y)\, 
\dfrac{\partial}{\partial n_y}\, p_{1-u}(b,y) \,  du, \qquad y\in\partial O.
\]
\end{enumerate}
To prove \eqref{hariya}, the author of \cite{hari} assumes that $\partial\Omega$
is smooth, and in particular that
\begin{enumerate}
\item for each $t>0$ and $y\in O$, $p_t(\cdot,y)$ is $C^1$ up to the boundary
\item the restriction to $\partial O$ of harmonic functions on $O$, and $C^1$ up to the boundary, are dense in $C(\partial O)$
\end{enumerate}
see \cite[Remarks 1.1 and 1.2]{hari}.
Under this assumption the law of $(\tau(B),B_{\tau (B)})$ is given for
$a\in O$ by 
\[
\bbP_a(\tau(B)\in dt,B_{\tau (B)} \in dy) = \dfrac{1}{2}\, \dfrac{\partial}{\partial n_y}\, p_t(a,y) \, \sigma(dy) \, dt,
\]
where $\partial/\partial n_y$ denotes the normal derivative at $y \in \partial O$,
see \cite[formula (1.4)]{hari}. 

We want to deduce from \eqref{hariya} the following integration by parts formula for $\nu$.
We set for all bounded Borel $F:H\mapsto\R$
\begin{equation}\label{Sigma}
\int F(w)\, \Sigma(y,dw) := \frac1Z
 \, \E_{a,y,b}\left[F(w) \, e^{-U(w)}\right] \, \lambda(y).
\end{equation}
\begin{prop}
For all $F \in \mathcal{FC}^1$ 
\begin{equation}\label{ibpnu}
\begin{split}
\int \partial_h F \, d\nu = & -\int \langle h'',x\rangle \, F\, d\nu +
\int \langle h,\partial_0\varphi\rangle F\, d\nu 
\\ & - \int_{\partial O}\sigma(dy)\int n(y)\cdot h(S_w)\, F(w)\ \Sigma(y,dw)
\end{split}
\end{equation}
\end{prop}
\begin{proof}
If $F \in \mathcal{FC}^{1}$, then we apply \eqref{hariya} with the function $Fe^{-U_n}$, where $U_n$ is defined
in \eqref{U_n}, and we obtain
\[
\begin{split}
\int_K \partial_h F\, e^{-U_n} \, d\mu = & -\int_K \langle h'',x\rangle F\, e^{-U_n} \, d\mu +
\int_K \langle h,\partial\Phi_n\rangle F\, e^{-U_n} \, d\mu  \\ & -\int_{\partial O}\sigma(dy) \, \E_{a,y,b}\left[n(y)\cdot h(S)\, F\, e^{-U_n}\right] \lambda(y).
\end{split}
\]
The dominated convergence theorem and Lemma \ref{yosid} provide the desired result.
\end{proof}

\section{Existence of weak solutions}\label{exweso}

By Theorem \ref{exmarkov} we have a Markov process associated with the Dirichlet form
$\cE$ defined by \eqref{e^}, but we still have to show that it is a solution
of \eqref{1}. In particular, we have the process $X$, namely the function $u$,
but not the reflection measure $\eta$. The aim of this section is to construct
$\eta$ and obtain a weak solution of equation \eqref{1}, in particular to prove
the following
\begin{prop}\label{exwe}
For all $x\in K$ there exists a weak solution $(u,\eta,W)$ of equation
\eqref{1}.
\end{prop}
We are going to use Fukushima's theory \cite{fot} and in particular the
powerful correspondence between positive continuous additive functionals (PCAF)
and {\it smooth} measures, i.e. positive measures which do not charge sets with zero
capacity. This theory is the content of \cite[Chapters 4 and 5]{fot}, to which we refer for all details. 

We explain now why construction of a solution of \eqref{1} is not trivial, despite
all information we already have. Since the main difficulty comes from the reflection term,
let us suppose for simplicity that $\varphi\equiv0$ and therefore,
recalling the definition \eqref{Phi} of $\Phi$, we have $\Phi\equiv 0$ on $\overline O$ 
and $\Phi\equiv +\infty$ on $\R^d\setminus\overline O$. 
Then the Yosida approximation $\Phi_n$ of $\Phi$ is equal to
\[
\Phi_n(y)=nd^2(y,\overline O):=n\inf_{z\in \overline O}\|y-z\|^2, \qquad y\in \R^d
\]
and its differential is $\partial \Phi_n(y)=2nd(y,\overline O)\frac{y-p(y)}{|y-p(y)|}$, where
$p(y)\in \overline O$ minimizes the distance from $y$, i.e. $d(y,\overline O)=\|y-p(y)\|$.
Therefore, \eqref{appro} becomes
\[
\frac{\partial u_n}{\partial t}=\frac 12
\frac{\partial^2 u_n}{\partial \theta^2} - nd(u_n,\overline O)\frac{u_n-p(u_n)}{|u_n-p(u_n)|}
 + \dot W.
 \]
 By Theorem \ref{stability}, we already know that $u_n$ converges weakly to a process $u$.
 In all papers on reflected SPDEs with real values, one uses at some point that if $O\subset\R$ is an interval, then $\frac{y-p(y)}{|y-p(y)|}$ belongs to $\{\pm 1\}$ and is therefore locally constant. In other words
 one can decompose the non-linearity 
 \[
 nd(u_n,\overline O)\frac{u_n-p(u_n)}{|u_n-p(u_n)|}= \eta^+_n-\eta^-_n
 \]
 where $\eta^+_n,\eta^-_n\geq 0$ have well separated supports by the continuity
 of $u_n$. Moreover, it is not too difficult to obtain bounds on the total variation $\eta^+_n,\eta^-_n$, which yield tightness and therefore
convergence of $\eta^+_n,\eta^-_n$ as $n\to\infty$, as has been done in a number of papers, see \cite{nupa, dopa1, eddouk, deza, 
goud, debgoud} among others.

On the other hand, if $u_n\in\R^d$, then such a decomposition becomes impossible, since
the vector $\frac{y-p(y)}{|y-p(y)|}$ varies continuously in $\bbS^{d-1}$ and the process
\[
t\mapsto L_n(t):=\int_0^t \left[nd(u_n,\overline O)\frac{u_n-p(u_n)}{|u_n-p(u_n)|}\right](s,\theta)\, ds
\]
has no definite sign. Therefore, convergence of $u_n$ yields some form of convergence of $L_n$
to a process $L$, but, without control on the total variation of $L_n$, we cannot even guarantee that
$L$ has bounded variation, a necessary condition if we want to obtain a measure $\eta$ in
equation \eqref{1}. This is the main reason why the approaches available in
the literature do not work in our case.

\subsection{Dirichlet forms and Additive Functionals}
We recall here the basics of potential theory which are needed in what follows,
referring to \cite{fot} and \cite{maro} for all proofs. By Theorem \ref{exmarkov},
the Dirichlet form $(\cE,D(\cE))$ has an associated Markov process, which is
also a Hunt process. Therefore, by \cite[Theorem IV.5.1]{maro}, the Dirichlet form
is {\it quasi-regular}, i.e. it can be embedded into a regular Dirichlet form;
in particular, the classical theory of \cite{fot} can be applied.
Moreover, the important {\it absolute continuity condition} \eqref{eq:absocon}
allows in the end to get rid of exceptional sets: see for
instance \cite[Theorem 4.1.2 and formula (4.2.9)]{fot}.

We denote by $\cF_\infty^\lambda$ (resp.  $\cF_t^\lambda$) the completion of $\cF_\infty^0$ (resp.
completion of $\cF_t^0$ in $\cF_\infty^\lambda$) with
respect to ${\mathbb P}_\lambda$ and we set $\cF_\infty:=\cap_{\lambda\in{\cP}(K)} \, \cF_\infty^\lambda$, $\cF_t:= \cap_{\lambda\in{\cP}(K)} \, \cF_t^\lambda$,
where ${\cP}(K)$ is the set of all Borel probability measures on $K$. 

\subsubsection{Capacity}
Let $A$ be an open subset of $H$, we define by $\cL_A:=\{u \in D(\cE): u \geq 1$, $\nu$-a.e. on $A\}$. Then we set
\[
{\rm Cap}(A)=\left\{ \begin{array}{ll}
\inf\limits_{u\in\cL_A}\cE_1(u,u), & \cL_A \neq \emptyset, \\
+\infty & \cL_A = \emptyset,
\end{array} \right.
\] 
where $\cE_1$ is the inner product on $D(\cE)$ defines as follow
\[
\cE_1(u,v) = \cE(u,v)+\int_H u(x)\, v(x) \, d\nu, \quad u,v\in D(\cE).
\]
For any set $A\subset H$ we let 
\[
{\rm Cap}(A) = \inf\limits_{B \ {\rm open}, A \subset B\subset H} {\rm Cap}(B)
\]
A set $N\subset H$ is {\it exceptional} if ${\rm Cap}(N)=0$.

\subsubsection{Additive functionals}
By a Continuous Additive Functional (CAF) of $X$, we mean a
family of functions $A_t:E\mapsto {\mathbb R}^+$, $t\geq 0$, such
that:
\begin{itemize}
\item[(A.1)] $(A_t)_{t\geq 0}$ is $({\cF}_t)_{t\geq 0}$-adapted
\item[(A.2)] There exists a set $\Lambda\in{\cF}_\infty$ and
a set $N\subset K$ with ${\rm Cap}(N)=0$ such that
${\mathbb P}_x(\Lambda)=1$ for all $x\in K\setminus N$,
$\theta_t(\Lambda)\subseteq \Lambda$ for all $t\geq 0$, and for all
$\omega\in \Lambda$: $t\mapsto A_t(\omega)$ is 
continuous, $A_0(\omega)= 0$ and for all $t,s\geq 0$:
\[
A_{t+s}(\omega) \, = \, A_s(\omega)+A_t(\theta_s\omega),
\]
where $(\theta_s)_{s\geq 0}$ is the time-translation semigroup on
$E$. 
\end{itemize}
Moreover, by a Positive Continuous Additive Functional (PCAF) of $X$ we mean a CAF of $X$
such that:
\begin{itemize}
\item[(A.3)] For all $\omega\in \Lambda$: $t\mapsto A_t(\omega)$ is 
non-decreasing.
\end{itemize}
Two CAFs 
$A^1$ and $A^2$ are said to be equivalent if
\[
{\mathbb P}_x \left( A^1_t=A^2_t \right) \, = \, 1, 
\quad \forall t>0, \ \forall x\in K\setminus N.
\] 
If $A$ is a linear combination of PCAFs of 
$X$, the Revuz-measure of $A$ 
is a Borel signed measure $\Sigma$ on $K$ such that:
\[
\int_{K} \varphi\, d\Sigma \, = \, 
\int_{K} {\mathbb E}_x\left[
\int_0^1 \varphi(X_t)\, dA_t \right]\,
\nu(dx), \quad \forall \varphi\in C_b(K).
\] 

\subsubsection{The Fukushima decomposition}
Let $h\in C^2_0((0,1);\R^d)$, and set $U:K\mapsto {\mathbb R}$,
$U(x):=\langle x,h \rangle$. By Theorem \ref{exmarkov}, the Dirichlet 
Form $({\cE},D({\cE}))$ is quasi-regular. Therefore
we can apply the Fukushima decomposition, as it is stated in
Theorem VI.2.5 in \cite{maro}, p. 180:
for any $U\in {\rm Lip}(H)\subset D({\cE})$,
we have that there exist an exceptional set $N$, a Martingale Additive Functional of finite
energy $M^{[U]}$ and a Continuous Additive Functional of zero energy 
$N^{[U]}$, such that for all $x\in K\setminus N$: 
\begin{equation}\label{tipler}
U(X_t)-U(X_0) \, = \, M^{[U]}_t + N^{[U]}_t, \quad t\geq 0, \
{\mathbb P}_x-{\rm a.s.}
\end{equation}

\subsubsection{Smooth measures}
We recall now the notion of smoothness for a positive Borel measure $\Sigma$ on $H$,
see \cite[page 80]{fot}. A positive Borel measure $\Sigma$ is {\it smooth} if
\begin{enumerate}
\item $\Sigma$ charges no set of zero capacity
\item there exists an increasing sequence of closed sets $\{F_k\}$ such that
$\Sigma (F_n) < \infty$, for all $n$ 
and $\lim\limits_{n\to \infty}{\rm Cap}(K-F_n)=0$ for all compact set $K$.
\end{enumerate}
By definition, a signed measure $\Sigma$ on $H$ is smooth if its total variation measure $|\Sigma|$ is smooth. That happens if and only if $\Sigma=\Sigma^1 - \Sigma^2$, where $\Sigma^1$ and $\Sigma^2$ are positive smooth measures, obtained from $\Sigma$ by applying the Jordan decomposition (see \cite[page 221]{fot}).

We recall a definition from \cite[Section 2.2]{fot}.
We say that a positive Radon measure $\Sigma$ on $H$ is {\it of finite energy}
if for some constant $C>0$
\begin{equation}\label{finiteenergy}
\int |v|\, d\Sigma\leq C\sqrt{\cE_1(v,v)},\qquad \forall \, v\in D(\cE)\cap C_b(H).
\end{equation}
If \eqref{finiteenergy} holds, then there exists an element $U_1\Sigma$ such that
\[
\cE_1(U_1\Sigma,v) = \int_H v\, d\Sigma, \qquad  \forall \, v\in D(\cE)\cap C_b(H).
\]
Moreover, by \cite[Lemma 2.2.3]{fot}, all measures of  finite energy are smooth.

Finally; by \cite[Theorem 5.1.4]{fot}, if $\Sigma$ is a positive smooth measure, then
there exists a PCAF $(A_t)_{t\geq 0}$, unique up to equivalence, with Revuz measure equal to $\Sigma$.

\subsection{The non-linearity}
We prove first that a.s. the function $(t,\theta)\mapsto |\partial_0 \varphi(u(t,\theta))|$
is in $L^1_{\rm loc}([0,T]\times\,]0,1[)$ for all $T\geq 0$. We start by the following
\begin{lemma}\label{nuphi}
For all $\delta\in(0,1/2)$ 
\[
\int \nu(dx)\left(
\int_\delta^{1-\delta} |\partial_0 \varphi(x_\theta)| \, d\theta
\right)^2<+\infty.
\]
\end{lemma}
\begin{proof}
We have for all $\theta\in[\delta,1-\delta]$, by the definition \eqref{nu} of $\nu$
\[
\begin{split}
\int \nu(dx)\,  |\partial_0 \varphi(x_\theta)|^2 \leq & \ \frac1Z \int \mu(dx)\,  |\partial_0 \varphi(x_\theta)|^2 \, \un{
\left\{x_\theta\in \overline O\right\}} = \frac1{C(\theta)} \int_O |\partial_0 \varphi(z)|^2 \, e^{-\frac{|z|^2}{2\theta(1-\theta)} }\, dz
\\ \leq & \ \frac1{C_\delta} \int_O |\partial_0 \varphi(z)|^2 \, dz<+\infty
\end{split}
\]
by \eqref{assumphi}. Since this quantity does not depend on $\theta\in[\delta,1-\delta]$, we have the desired result by H\"older's inequality:
\[
\int \nu(dx)\left(
\int_\delta^{1-\delta} |\partial_0 \varphi(x_\theta)| \, d\theta
\right)^2\leq\int \nu(dx)
\int_\delta^{1-\delta} |\partial_0 \varphi(x_\theta)|^2 \, d\theta
 <+\infty.
\]
\end{proof}
Now we obtain that
\begin{prop}\label{L1loc}
The functional
\[
C_t:=\int_0^t \int_\delta^{1-\delta} |\partial_0 \varphi(u(s,\theta))| \, d\theta
\, ds, \qquad t\geq 0,
\]
is a well-defined PCAF of $X$. In particular,
the function $(t,\theta)\mapsto |\partial_0 \varphi(u(t,\theta))|$
is in $L^1_{\rm loc}([0,T]\times\,]0,1[)$ for all $T\geq 0$, $\bbP_x$-a.s.
for all $x\in K\setminus N$ for some $N\subset K$ with ${\rm Cap}(N)=0$.
\end{prop}
\begin{proof} Setting
\[
F:H\mapsto[0,+\infty], \qquad F(w):=\int_\delta^{1-\delta} |\partial_0 \varphi(w_\theta)| \, d\theta,
\]
then by Lemma \ref{nuphi} $F\in L^2(\nu)$ and moreover we can write
$C_t=\int_0^t F(X_s)\, ds$, $t\geq 0$. Denoting
\[
R_1F(x):=\int_0^\infty e^{- t} \,
{\mathbb E}_x \, [F(X_t)]\, dt, \qquad
x\in K, 
\]
then it is well known that
\[
{\cE}_1(R_1F,v) \, = \, \int_H F\, v \, d\nu,
\quad \forall v\in D({\cE}),
\]
and therefore \eqref{finiteenergy} holds. Then, $F\, d\nu$
is smooth and the associated PCAF is $(C_t)_{t\geq 0}$.
\end{proof}

\subsection{The reflection measure}
We are going to apply \eqref{tipler} to $U^h(x):=\langle x,h\rangle$, $x\in H$,
with $h\in C^2_c((0,1);\R^d)$. Clearly $U^h\in {\rm Lip}(H)\subset D({\cE})$.
Our aim is to prove the following
\begin{prop}\label{decomposition}
There are an exceptional set $N$ and
a unique measure ${\eta}(ds,d\theta)$ on $[0,+\infty[\, \times[0,1]$ such that 
for all $x\in K\setminus N$, $\bbP_x$-a.s. for all $t\geq 0$
\begin{equation} \label{decomposition2}
\begin{split}
& N_t^{[U^h]} = \\
& = \int_0^t \int_0^1 h_{\theta}\cdot n(u(s,\theta)) \,{\eta}(ds,d\theta)
 + \frac12\int_0^t \langle h'', u_s\rangle ds - \frac12\int_0^t \langle h,\partial_0\varphi(u_s)\rangle ds 
\end{split}
\end{equation}
where $h\in C_c^{\infty}((0,1);\R^d)$, and {\rm Supp}$({\eta})\subset \{(t,\theta) \, | \, u(t,\theta)\in \partial O\}$.
\end{prop}
The main tools of the proof are the integration by parts formula \eqref{ibpnu} and
a number of results from the theory of Dirichlet forms in \cite{fot}.
We start by noticing that, by applying \eqref{tipler} to $U^h(x):=\langle x,h\rangle$, $x\in H$, we obtain, recalling the definition \eqref{Sigma} of $\Sigma(y,dw)$:
\begin{lemma}\label{constreta}
The process $N^{[U^h]}$ is a linear combination of PCAFs
of $X$, and its Revuz measure is $\frac 12\,\Sigma^h$, where
\begin{equation} \label{Sigmah}
\Sigma^h(dw):=\left( \langle w,h''\rangle - \langle \partial_0\phi(w),h \rangle \right) \cdot
\nu(dw) \, + \int_{\partial O} \sigma(dy)\, n(y)\cdot h(S_w)\,   
\Sigma(y,dw).
\end{equation}
\end{lemma}
\begin{proof}
The integration by parts formula \eqref{hariya} can be rewritten
as
\[
\cE(U^h,v)=\frac12\int v(w)\, \Sigma^h(dw), \qquad \forall
\, v\in D(\cE)\cap C_b(H).
\]
By \cite[Corollary 5.4.1]{fot}, this implies that $\frac12\Sigma^h$ is
the Revuz measure of $N^{[U^h]}$ and that $\Sigma^h$ is 
a smooth signed measure. By \cite[Theorem 5.4.2]{fot}, 
this implies that $N^{[U^h]}$ is $\bbP_x$-a.s. a bounded variation process
for all $x\in K$; moreover, the Jordan decomposition $\Sigma^h
=\Sigma^h_1-\Sigma^h_2$, with $\Sigma^h_i$ positive measures
concentrated on disjoint sets, corresponds to a decomposition
$N^{[U^h]}=N^h_1-N^h_2$ with $N^h_i$ a PCAF of $X$ with
Revuz measure $\frac12\Sigma^h_i$.

\begin{lemma}\label{totvar} For all $h\in C^2_c((0,1);\R^d)$, 
the total variation measure $|\Sigma^h|$ of 
$\Sigma^h$ is equal to
\[
|\Sigma^h|(dw)=\left| \langle w,h''\rangle - \langle \partial_0\phi(w),h \rangle \right| \cdot
\nu(dw) \, +  \int_{\partial O} \sigma(dy)\, \left|n(y)\cdot h(S_w)\right| \,  
\Sigma(y,dw).
\]
\end{lemma}
\begin{proof}
Now, we can
notice that
$\nu(dw)$ and $\int_{\partial O} \sigma(dy)\, \Sigma(y,dw)$ are mutually singular,
since the former measure is concentrated on trajectories not
hitting the boundary $\partial O$, and the latter on trajectories 
hitting $\partial O$. Therefore
\[
|\Sigma^h|(dw)=\left| \langle w,h''\rangle - \langle \partial_0\phi(w),h \rangle \right| \cdot
\nu(dw) \, + \left|\int_{\partial O} \sigma(dy)\, n(y)\cdot h(S_{\cdot}) \,  
\Sigma(y,\cdot)\right|(dw).
\]
Now, by considering the sets $A:=\{w: n(w(S_w))\cdot h(S_w)\geq 0\}$ and 
$B=\{w: n(w(S_w))\cdot h(S_w)< 0\}$, we can see that
\[
 \left|\int_{\partial O} \sigma(dy)\, n(y)\cdot h(S_{\cdot}) \,  
\Sigma(y,\cdot)\right|(dw) = \int_{\partial O} \sigma(dy)\, \left|n(y)\cdot h(S_w)\right| \,  
\Sigma(y,dw)
\]
and we have the desired result.
\end{proof}
By definition, the total variation measure $|\Sigma^h|$ is smooth, and therefore
so is the measure
\[
\int_{\partial O} \sigma(dy)\, \left|n(y)\cdot h(S_w)\right| \,  
\Sigma(y,dw),
\]
since it is non-negative and bounded above by $|\Sigma^h|$, for any $h\in C^2_c((0,1);\R^d)$. 
Let us now consider a non-negative $g\in C^2_c(0,1)$ and a basis $\{e_1,\ldots,e_d\}$ of $\R^d$. Then the measure
\[
\Lambda^g(dw):=\sum_{i=1}^d \int_{\partial O} \sigma(dy)\, g(S_w)\,\left|n(y)\cdot e_i\right| \, \Sigma(y,dw)
\]
is smooth since it is sum of smooth measures. For any interval
$I\, \Subset (0,1)$ we can set
\begin{equation}\label{gammaI}
\Gamma^{I}(dw):=\int_{\partial O} \sigma(dy)\, \un{I}(S_w) \,  
\Sigma(y,dw).
\end{equation}
Let $\kappa:=\min_{z\in\bbS^{d-1}}\sum_{i=1}^d|z\cdot e_i|$. By compactness,
$\kappa>0$ and therefore, for $g\in C^2_c(0,1)$ such that $g\geq \un{I}$,
we obtain $0\leq \Gamma^{I}\leq \Lambda^g/\kappa$. Hence, $\Gamma^{I}$
is smooth.

 In particular, if $\{I_n\}_n$ is any
countable partition of $(0,1)$ in intervals $I_n\Subset(0,1)$, then
we obtain that the finite measure
\begin{equation}\label{gamma1}
\Gamma(dw):=\Gamma^1(dw)=\sum_n \Gamma^{I_n}(dw)
\end{equation}
is also smooth and finite by its explicit expression. Now, for any $g\in C([0,1])$, the measure
\[
\Gamma^{g}(dw):= \int_{\partial O} \sigma(dy)\, g(S_w) \,  
\Sigma(y,dw),
\]
is also smooth, since $|g|\leq \|g\|_\infty\, 1$ implies
$0\leq |\Gamma^{g}|\leq \|g\|_\infty\,\Gamma^1$.
By \cite[Theorem 5.1.4]{fot}, there exists
a PCAF $(A^g_t)_{t\geq 0}$, unique up to equivalence, with Revuz measure equal to $\Gamma^g$, for any $g\in C([0,1])$. At the same time, for any interval $I\subseteq[0,1]$, there exists a PCAF $(A^I_t)_{t\geq 0}$, unique up to equivalence, with Revuz measure equal to $\Gamma^I$. Moreover
\begin{equation}\label{esti}
|A^g_t| \leq \|g\|_\infty\, A^1_t, \qquad \forall\, t\geq 0,
\end{equation}
%
since the positive finite measure 
$(\|g\|_\infty\cdot\Gamma^1-\Gamma^{g})(dx)$ is finite and smooth and is therefore
the Revuz measure of a PCAF, so that we can conclude by the linearity
of the Revuz correspondence.

We want now to prove that there exists a finite positive measure $\eta$ on
$[0,T]\times[0,1]$ such that 
\begin{equation}\label{eccoeta}
A^g_T = \int_{[0,T]\times [0,1]} g(\theta)\, \eta(ds,d\theta) , \qquad \forall \, T\geq 0, \, g\in C([0,1]).
\end{equation}
Let $(g_n)_n$ be a dense sequence in $C([0,1])$. By the Revuz correspondence we have $A^{g_n+g_m} = A^{g_m}+A^{g_n}$. 
Let $\Lambda = \bigcap_{n,m} \{A^{g_n+g_m} = A^{g_m}+A^{g_n}\}$, so that
$\bbP_x(\Lambda)=1$ for all $x\in K\setminus N$, where
$N$ is an exceptional set. By \eqref{esti}, we obtain that
 the map $C([0,1])\ni g\mapsto A^g_T$
is linear, continuous and if $g\geq 0$ then $A^g_T\geq 0$.
Then by the Riesz representation theorem there exists a Radon measure 
$A_T(d\theta)$ on $[0,1]$, such that
\[
A^g_T=\int_{[0,1]} g(\theta)\, A_T(d\theta), \qquad \forall\, g\in C([0,1]).
\]
Moreover, $A_{t}(d\theta)$ satisfies
$0\leq A_{t}(d\theta)\leq A_T(d\theta)$  for $0\leq t\leq T$. Therefore $A_{t}\ll A_{T}$.
By the Radon-Nykodim theorem we have
\[
A_{t}(B) = \int_B C_{t}(y) {A}_T(dy)
\]
where $C_{t}\in L^1(A_{T}(d\theta))$.

Now, the problem is that $C_{t}(y)$ is defined for
${A}_T$-a.e. $y$, and the set of definition might depend on
$t$. We must show that it is possible to find a version of
$(C_{t})_{0\leq t\leq T}$ defined on the same set of full $A_T$-measure.

We claim that for ${A}_T$-a.e. $\theta$, 
$t\mapsto C_{t}(\theta)$ is equal to a c\`adl\`ag function. Indeed, let $(q_n)_n$ be a dense sequence in $[0,T]$ and
set
\[
{\Lambda}:=\bigcap\limits_{n} \left\{ \theta: C_{q_n}(\theta)\leq C_{t}(\theta),\ \forall \, t\in (q_m)_m,
\ q_n\leq t, \ C_{q_n}(\theta) = \lim\limits_{s\in(q_m)_m\downarrow q_n} C_{s}(\theta) \right\}.
\]
Notice that $A_T(\Lambda^c)=0$.
$C_{\cdot}$ is $d{A}$ a.e well defined. We denote by $\tilde{C}_{\cdot}(\cdot)$ the function defined on $[0,T]\times[0,1]$ by 
\[
\tilde{C}_{t}(\theta) := \lim\limits_{s\in(q_n) \downarrow t} C_{s}(\theta), \quad (t,\theta)\in[0,T[\, \times\Lambda, \qquad
\tilde{C}_{T}:=C_{T},
\]
and $\tilde{C}_{t}(\theta):=0$ if $t<T$ and $\theta\notin \Lambda$. By continuity of $t\mapsto A_{0,t}(B)$, we obtain
that
\[
A_{t}(B) = \int_B \tilde C_{t}(y)\, {A}_T(dy).
\]
Moreover $\tilde{C}_{\cdot}(\theta)$ is c\`adl\`ag and non-decreasing and measurable,
so that there exists a measurable kernel $(\gamma_y(B), \, y\in [0,1], \, B\in\cB([0,T]))$, such that 
$\tilde{C}_{t}(y)-\tilde{C}_{s}(y)=\gamma_y({]s,t]})$ and therefore
\[
\begin{split}
{A}_t(B)-{A}_s(B) & =  
\int_B\left(\tilde{C}_{t}(y)-\tilde{C}_{s}(y)\right) {A}_T(dy)\\
               & = \int_B \gamma_y(]s,t]) \, {A}_T(dy), \qquad t,s\in [0,T].
\end{split}
\]
Therefore, the measure $\eta(ds,dy):=\gamma_y(ds) \, {A}_T(dy)$
on $[0,T]\times[0,1]$ satisfies \eqref{eccoeta}.

\medskip
Now we have to show that the measure $\eta$ satisfies {\rm Supp}$({\eta})\subset \{(t,\theta) \, | \, u(t,\theta)\in \partial O\}$ and \eqref{decomposition}.
We set 
\[
F:C\left([0,1];\overline O\right)\mapsto\R, \quad F(w):=
\un{\{w(\theta)\notin\partial O, \ \forall\, \theta\in[0,1]\}}
\]
and
\[
L_t:=
\int_0^t F(X_s) \, \eta(ds\times[0,1]),
\qquad t\geq 0.
\]
Then, by \cite[Theorem 5.1.3]{fot}, $(L_t)_{t\geq 0}$ is a PCAF of $X$
with Revuz measure given by $\frac12f(w)\cdot\Gamma(dw)$, see \eqref{gamma1}.
On the other hand, $\Gamma(\{w: w(\theta)\notin\partial O, \ \forall\, \theta\in[0,1]\})=0$ by the very definition of $\Gamma$; indeed, $\Sigma(y,dw)$ is the
law of a process which visits a.s. $y\in\partial O$ at same time in $[0,1]$.

Therefore, by the one-to-one correspondence between PCAFs and positive
smooth measures, see \cite[Theorem 5.1.3]{fot}, we conclude that 
$f(w)\cdot\Gamma(dw)\equiv 0$ and therefore $L\equiv 0$. Thus,
for $\eta(ds\times[0,1])$-a.e. $s$, $u(s,\cdot)$ visits $\partial O$
at some $\theta\in(0,1)$, and in particular $S(u(s,\cdot))$, i.e.
the smallest such $\theta$, is in $(0,1)$.

Let us now notice that, again by \cite[Theorem 5.1.3]{fot} and by 
\eqref{eccoeta},
for any bounded Borel $G:C\left([0,1];\overline O\right)\mapsto\R$
and for any bounded Borel $g:[0,1]\mapsto\R$,
the process
\[
t\mapsto \int_0^t G(X_s) \, dA^g_s=
\int_{[0,t]\times[0,1]} G(X_s) \, g(\theta)\, \eta(ds,d\theta)
\]
is a PCAF of $X$ with Revuz measure $\frac12G(w)g(S_w)\,\Gamma(dw)$.
Therefore for any bounded Borel $G:C\left([0,1];\overline O\right)\times[0,1]
\mapsto\R$, the process
\[
t\mapsto \int_{[0,t]\times[0,1]} G(X_s,\theta)\, \eta(ds,d\theta)
\]
is a PCAF of $X$ with Revuz measure $\frac12G(w,S_w)\,\Gamma(dw)$. In particular,
if we choose $G(w,\theta):=\un{\{w(\theta)\notin\partial O\}}$, then
the process
\[
t\mapsto \int_{[0,t]\times[0,1]} \un{\{u(s,\theta)\notin\partial O\}} \, \eta(ds,d\theta)
\]
is a PCAF of $X$ with Revuz measure $\frac12\un{\{w(S_w)\notin\partial O\}}\Gamma(dw)\equiv 0$.

Therefore, by the one-to-one correspondence between PCAFs and positive
smooth measures, see again \cite[Theorem 5.1.3]{fot}, we conclude that
$\eta(\{(s,\theta):\, u(s,\theta)\notin\partial O\})=0$, i.e. {\rm Supp}$({\eta})\subset \{(s,\theta) \, | \, u(s,\theta)\in \partial O\}$.

It remains to show \eqref{decomposition}. 
We recall that, by \eqref{eccoeta}, for all Borel $I\subseteq [0,1]$,
the process $t\mapsto \eta([0,t]\times I)$ is a PCAF of $X$ with Revuz measure
$\frac12\Gamma^I$, see \eqref{gammaI}. Now, it is enough to notice that the 
CAF in the right hand side of \eqref{decomposition} has Revuz measure $\frac12\Sigma^h$, given by \eqref{Sigmah}. Since $N^{[U^h]}$ has the same Revuz measure,
then by the one-to-one correspondence between PCAFs and positive
smooth measures, $N^{[U^h]}$ and the 
CAF in the right hand side of \eqref{decomposition} are
equivalent.
\end{proof}

\subsection{Identification of the noise term}
We deal now with the identification of $M^{[U^h]}$ with
the integral of $h$ with respect to a space-time white noise.
\begin{prop}\label{decomposition3}
There exists a Brownian sheet $(W(t,\theta), \, t\geq 0, \theta\in[0,1])$,
such that
\begin{equation}
M^{[U^h]}_t = \int_0^t \int_0^1 h_{\theta} \, W(ds,d\theta), \qquad h \in H.
\end{equation}
\end{prop}
\begin{proof}
We recall that, for $U\in D(\cE)$, the process $M^{[U]}$ is a continuous
martingale, whose quadratic variation 
$(\langle M^{[U]}\rangle_t)_{t\geq 0}$
is a PCAF of $X$ with Revuz measure $\mu_{\langle M^{[U]}\rangle}$ given by
the formula
\begin{equation}\label{quadratic}
\int f\, d\mu_{\langle M^{[U]}\rangle} = 2\cE(Uf,U)-\cE(U^2,f), \qquad
\forall \, f\in D(\cE)\cap C_b(H),
\end{equation}
see \cite[Theorem 5.2.3]{fot}. Now, if we apply this formula
to $U^h(x)=\langle x,h\rangle$, then we obtain 
\[
\int f\, d\mu_{\langle M^{[U^h]}\rangle} = \|h\|^2\int f\,d\nu, \qquad
\forall \, f\in D(\cE)\cap C_b(H).
\]
Therefore, the quadratic variation $\langle M^{[U^h]}\rangle_t$ is equal
to $\|h\|^2t$ for all $t\geq 0$, and, by L\'evy's Theorem, $(M^{[U^h]}\cdot\|h\|^{-1})_{t\geq 0}$ is a Brownian motion. Moreover, the parallelogram law,
if $h_1,h_2\in H$ and $\langle h_1,h_2\rangle=0$, then the
quadratic covariation between $M^{[U^{h_1}]}$ and $M^{[U^{h_2}]}$ is equal to
\[
\langle M^{[U^{h_1}]}, M^{[U^{h_2}]}\rangle_t = t \, \langle h_1,h_2\rangle,\qquad
t\geq 0.
\] 
Therefore, $(M^{[U^h]}_t, t\geq 0, h\in H)$ is a Gaussian process
with covariance structure
\[
\E_x\left(M^{[U^{h_1}]}_t \, M^{[U^{h_2}]}_s \right) = 
s\wedge t\,  \langle h_1,h_2\rangle.
\]
If we define $W(t,\theta):=M^{[U^{h}]}_t$ with
$h:=\un{[0,\theta]}$, $t\geq 0$, $\theta\in[0,1]$, then $W$ is the
desired Brownian sheet.
\end{proof}

\subsection{From $K\setminus N$ to $K$}
We have so far proved existence of an exceptional set $N$ such that
for all $x\in K\setminus N$ there is a weak solution of equation
\eqref{1}. We show now how to construct a weak solution for $x\in N$.

Let $x\in N$.
By the absolute continuity relation \eqref{eq:absocon}, we have that
$\bbP_x$-a.s. $X_\gep\in K\setminus N$ for $\gep>0$, since $\nu(K\setminus N)=1$. Therefore,
we can set for all $\omega\in E$ and $0< \gep\leq s\leq t$
\[
\eta^\gep([s,t]\times I)(\omega):=\eta([s-\gep,t-\gep]\times I)(\theta_\gep\omega),
\]
where $(\theta_t)_{t\geq 0}$ is the time-translation operator of $E$. Then
$\gep\mapsto\eta^\gep([s,t]\times I)$ is monotone non-increasing, since
\[
\eta^\gep([s,t]\times I)(\omega)-\eta^\delta([s,t]\times I)(\omega)
=\eta^{\delta-\gep}([s,t]\times I)(\theta_\gep\omega), \qquad
0<\gep<\delta.
\]
As $\gep\downarrow 0$, we obtain existence of a $\sigma$-finite
measure $\eta(ds,d\theta)$ on $]0,T]\times[0,1]$, which satisfies the
required properties. A similar argument works for the non-linear 
part. The proof of Proposition \ref{exwe} is concluded.

\section{Pathwise uniqueness and strong solutions}

We prove that equation \eqref{1} has a pathwise unique solution. This follows the lines of \cite{nupa}.
By a Yamada-Watanabe type result from \cite{kurtz}, pathwise uniqueness and existence of weak solutions imply existence and uniqueness of strong solutions and uniqueness in law. 

\subsection{Pathwise uniqueness}

\begin{prop}\label{uniuni}
Pathwise uniqueness holds for equation \eqref{1}.
\end{prop}
\begin{proof}
Let $(u^1,{\eta}^1,W)$ and $(u^2,{\eta}^2,Z)$ be two weak solutions of  \eqref{1}, we denote 
 \[
 z := u^1-u^2, \qquad \pi(ds, d\theta) = n(u^1(s,\theta))\cdot{\eta}^1(ds,d\theta)-n(u^2(s,\theta))\cdot{\eta}^2(ds,d\theta),
 \] 
 so for $h \in C^2_c((0,1)\times[0,T];\R^d)$ and $0<\gep\leq T$, denoting
 $\partial_s=\frac{\partial}{\partial s}$ and $\partial_\theta^2=\frac{\partial^2}{\partial\theta^2}$, 
\begin{equation}\label{uniqueness}
\begin{split}
\langle h_T,z_T\rangle-\langle h_\gep,z_\gep\rangle = &\ \frac12 \int_\gep^T \langle h'', z_s\rangle \,ds - \frac12\int_\gep^T\langle h_s, \partial_0 \phi (u^1_s) -\partial_0\phi(u^2_s)\rangle ds 
\\
& +  \int_\gep^T\int_0^1 h(s,\theta)\cdot\pi(ds,d\theta)+ \int_\gep^T \langle 
\partial_s{h}_s, z_s\rangle ds.
\end{split}
\end{equation}
Let $\zeta$ be an infinitely differentiable even function, with support contained in $[-1,1]$, such that $\int_{[-1,1]}\zeta(x)dx = 1$ and $\sum_{i,j}\zeta(x_i-x_j)y_iy_j\geq 0$ for any $(x_i)_{i\leq n}$ and $(y_i)_{i\leq n}\in \bbR^n$, $n\in \bbN$. Let  $\psi$ be an 
infinitely differentiable function with compact support, we consider now the function $h_{n,m}$ defined by
\[
h_{n,m}:=((z\psi) \textit{\textasteriskcentered} \zeta_{n,m})\psi
\]
where $\zeta_n(x):=n\zeta(nx)$ and $\zeta_{n,m}(t,\theta):=\zeta_n(t)\zeta_m({\theta})$. We will study the asymptotic behaviour of each term in \eqref{uniqueness} substituting $h$ by $h_{n,m}$. First we have
\[
\lim\limits_{n,m}\, \langle h_{n,m}(t), z(t)\rangle = \|z(t)\,\psi\|^2.
\]
Next 
\[
\int_\gep^T \langle \partial_s{h}_{n,m}(s), z(s)\rangle\, ds = \int_\gep^T\int_{(t-1/n)^+}^{t+1/n}\zeta_n'(t-s)\,\Gamma_m(s,t)\,ds\,dt
\]
where $\Gamma_m$ is a symmetric function of $(s,t)$, defined by 
\[
\Gamma_m(s,t):=\int_0^1\int_0^1z(s,\theta)\cdot z(t,\upsilon) \, \psi(\theta)\,\zeta_m(\upsilon - \theta)\,\psi(\upsilon) \,d\theta \,d\upsilon.
\]
As $\zeta'(s)=-\zeta'(-s)$ the integral 
\[
\int_\gep^T\int_{\max(t-1/n,\gep)}^{\min(t+1/n,T)}\zeta_n'(t-s)\,\Gamma_m(s,t)\,ds\,dt
=\int_{[\gep,T]^2} \un{\{|t-s|\leq 1/n\}}\,
\zeta_n'(t-s)\,\Gamma_m(s,t)\,ds\,dt
\]
vanishes. Therefore if $1/n\leq \gep$ then as $n\to+\infty$
\[
\begin{split}
& \left|\int_\gep^T \langle \partial_s{h}_{n,m}(s), z(s)\rangle \,ds\right|  \leq  \left|\int_{T-1/n}^Tdt\int_T^{t+1/n}ds\, \zeta_n'(t-s)\,\Gamma_m(s,t) \right| \\
 & + \left|\int_{\gep}^{\gep+1/n}dt\int_{t-1/n}^\gep ds\, \zeta_n'(t-s)\,\Gamma_m(s,t) \right| \leq  \dfrac{K}{n} \to 0.
 \end{split}
\]
Now, because of the properties of $(u_i,\eta_i)$
\[
\begin{split}
& \lim\limits_{n,m} \int_\gep^T\int_0^1 h_{n,m}(s,\theta)\, \pi(ds,d\theta) =  \int_\gep^T\int_0^1 \psi(\theta) \, z(s,\theta) \cdot \pi(ds,d\theta)  \\
& =  -\int_\gep^T\int_0^1 \psi^2(\theta)\left\{
 u^2(s,\theta)\cdot n(u^1(s,\theta))\, {\eta}^1(ds,d\theta) + u^1(s,\theta)\cdot n(u^2(s,\theta))\, {\eta}^2(ds,d\theta) \right\} \\
&\leq  0.
\end{split}
\]
By the convexity of $\phi$  we have 
\[
\lim\limits_{n,m}\int_\gep^T\langle h_{n,m}(s), \partial_0 \phi (u^1_s) -\partial_0\phi(u^2_s)\rangle \, ds \geq 0.
\]
For the last term, we notice that $\int_\gep^T \langle \partial^2_\theta h_{n,m}, z_s\rangle ds \to \int_\gep^T \langle \partial^2_\theta h_{n}, z_s\rangle ds$ when $m \to \infty$. We first suppose that $z$ is smooth, integrating by parts $\langle\partial^2_\theta h_{n}(s),z(s)\rangle$ we obtain
\[
\langle\partial^2_\theta h_{n}(s),z(s)\rangle \leq 
\langle( z\psi) \textit{\textasteriskcentered} \zeta_{n},\psi''\, z(s)\rangle +\langle( z\psi') \textit{\textasteriskcentered} \zeta_{n},\psi'\, z(s)\rangle.
\]
Moreover we obtain the same inequality for $z$ approximating $z$ with smooth functions.
As a result
\[
\liminf\limits_{n} \lim\limits_{m}\int_\gep^T \langle \partial^2_\theta h_{n,m}, z_s\rangle ds\leq \dfrac{1}{2}\int_\gep^T\int_0^1 |z_s|^2\, (\psi^2)'' ds
\]
Finally, we have obtained
\[
\int_0^1 \left(z^2(T,\theta)-z^2(\gep,\theta)\right)\psi^2(\theta)\,d\theta \leq \dfrac{1}{2}\int_\gep^T\int_0^1z^2(s,\theta)\, (\psi^2)''(\theta)\, ds\, d\theta
\]
and letting $\gep\to 0$
\[
\int_0^1 z^2(T,\theta)\, \psi^2(\theta)\,d\theta \leq \dfrac{1}{2}\int_0^T\int_0^1z^2(s,\theta)\,(\psi^2)''(\theta)\, ds\, d\theta.
\]
The rest of the proof consists of choosing a judicious expression for $\psi$, which can be done as at the end of the proof of uniqueness in \cite{nupa}. Finally, we obtain that that $z\equiv 0$ and ${\eta}^1 = {\eta}^2$.
\end{proof}

\subsection{Strong solutions}
Until now we have dealt with weak solutions. Now we show that
all weak solutions are in fact strong. 

We recall that a weak solution is given by a triple $(u,\eta,W)$.
We set $\cX:=(u,\eta)$
and $\cY:=W$. In the notation of \cite{kurtz}, equation \eqref{1}
can be interpreted as a relation $\Gamma(\cX,\cY)=0$ with
$\Gamma:S_1\times S_2\mapsto\R$ a Borel function defined
on the product of two Polish spaces $S_1$ and
$S_2$, for which
{\it pathwise} (or {\it pointwise}) uniqueness holds by Proposition
\ref{uniuni}. 
Therefore, by \cite[Lemma 2.7]{kurtz}, 
any weak solution of \eqref{1} is also strong. This concludes the
proof of Theorem \ref{main}.

\section{The reflection measure}
We want now to prove Theorem \ref{main2}, following the approach
of \cite{za}.
  Let $I\subseteq[0,1]$ be a Borel set.  
Denote by $\psi_I$ the indicator
function of the set $\{x\in K: x(\theta)\notin\partial O, 
\, \forall \theta\in[0,1]\backslash I\}$. The key point is the
following formula: for all $F\in C_b(H)$
\begin{equation}
\label{keypoint}
\int_{\partial O} \sigma(dy) 
\int \psi_I (w)\, F(w)\ \Sigma(y,dw)
= \int_{\partial O} \sigma(dy)  \int F(w)\, \un{I}(S_w)\, 
\Sigma(y,dw).
\end{equation}
By the definition of $\psi_I$, this follows because 
$\Sigma(y,dw)$-a.s. $S_w$ is the only $\theta\in[0,1]$
such that $w_\theta\in\partial O$. 
Let $A_t:=\eta([0,t]\times[0,1])$, $t\geq 0$.
We consider the following PCAF of $X$:
\[
(\psi_I \cdot A)_t \, := \, \int_0^t \psi_I(X_s) \, 
dA_s, \quad t\geq 0.
\]
Its Revuz measure is
\[
\frac12\,\psi_I(w)\int_{\partial O}\sigma(dy) \, \Sigma(y,dw).
\] 
In particular, by (\ref{keypoint}):
\begin{eqnarray*}
& &
\int_{K} {\mathbb E}_x \left[ \int_0^1
\left[ F\psi_I \right](X_s)\,
dA_s \right]
\, \nu^F(dx)
\\ \\ & & = \, \frac12\,\int_{\partial O}\sigma(dy)\int\left[ F\psi_I \right](w)\, \Sigma(y,dw)
\, = \, \frac12\,\int_{\partial O}\sigma(dy)\int F(w)\, \un{I}(S_w)\, \Sigma(y,dw)
\end{eqnarray*}
which is the Revuz measure of $A^{1_I}$, see \eqref{eccoeta}. By Theorem 5.1.6 in
\cite{fot}, we obtain that $A$ and $A^{1_I}$ are 
in fact equivalent as PCAFs of $X$, i.e. for all $x\in K$:
\begin{equation}
\label{strum}
\eta([0,t],I) \, = \, \int_0^t
\psi_I(X_s)\, \eta(ds,[0,1])
\quad \forall t\geq 0, \ \bbP_x{\rm -a.s.}
\end{equation}
Fix $x\in K$. We consider regular conditional
distributions $(t,J)\mapsto\gamma(t,J)$
of $\eta$ on $[0,\infty)\times[0,1]$, w.r.t. the Borel map $(t,\theta)\mapsto t$, where $t\geq 0$, $J\subseteq[0,1]$ Borel. In other words,
we obtain a $\sigma$-finite
measurable kernel $(t,J)\mapsto\gamma(t,J)$ such that:
\begin{equation}
\label{struc}
\eta([t,T],J)
\, = \, \int_t^T \gamma(s,J) \, \eta(ds,
[0,1])
\end{equation}
for all $J\subset[0,1]$ and $0\leq t\leq T<\infty$.
By (\ref{strum}) and (\ref{struc}) there exists
a measurable set $S\subseteq{\mathbb R}^+$ such that a.s.:
\[
\eta \left (\left[{\mathbb R}^+\backslash S \right]
\, \times \, [0,1] \right) \, = \, 0, \quad
{\rm and \ for \ all} \ s\in S:
\ \gamma(s,[0,1]) \, > \, 0, 
\]
\begin{equation}\label{strup}
\gamma(s,[a_n,b_n]) \, = \, \psi_{[a_n,b_n]}
(X_s), \quad \forall a_n,b_n\in{\mathbb Q}\cap
[0,1].
\end{equation}
Notice that, since $\psi_I$ is an indicator function, the right hand
side of (\ref{strup}) can assume only the values $0$ and $1$. Therefore
the measure $I\mapsto\gamma(s,I)$ takes only the values $0$ and $1$ on
all intervals $I$ with rational extremes in $[0,1]$, and the value $1$
is assumed, since $\gamma(s,[0,1]) > 0$. Then
$\gamma(s,\,\cdot\,)$ is a Dirac mass at some point
$r(s)\in[0,1]$.

Let now $s\in S$ and $q_n,p_n\in{\mathbb Q}$, such that
$q_n\uparrow r(s)$, $p_n\downarrow r(s)$. Set
$I_n:=[q_n,p_n]$: then 
\[
1=\gamma(s,I_n)=\psi_{I_n}(X_s),
\]
which, by the definition of $\psi_{I_n}$,
means $u(s,\theta)\notin\partial O$ for all $\theta\in[0,1]
\backslash I_n$; moreover
\[
0=\gamma(s,[0,1]\backslash\{r(s)\})
=\psi_{[0,1]\backslash\{r(s)\}}(X_s), 
\]
so that
$u(s,r(s))\in\partial O$.
Therefore, $r(s)$ is the unique
$\theta\in[0,1]$ such that $u(s,\theta)\in\partial O$.
Finally, since the support of $\eta$ is contained in $\{(t,\theta):
u(t,\theta)\in\partial O\}$ and a.s.
\[
\left(S\times[0,1]\right) \cap \{(t,\theta):
u(t,\theta)\in\partial O\} \, = \, 
\{(s,r(s)):s\in S\}:=\cS,
\]
then 
$\eta \left(
({\mathbb R}^+\times[0,1])\backslash\cS
 \right) \, = \, 0$. This concludes the proof of Theorem
 \ref{main2}.

\bigskip\noindent
{\bf Acknowledgments.} This is part of the PhD thesis of the author, who would like to thank his advisor Lorenzo Zambotti for introducing this subject, for his useful advice and for his encouragement.

\end{document}